\documentclass[11pt,a4paper]{article}

\usepackage[T1]{fontenc}
\usepackage[utf8]{inputenc}
\usepackage{lmodern}
\usepackage[margin=1in]{geometry}
\usepackage{amsmath,amssymb,amsthm,mathtools}
\usepackage{xurl}
\usepackage{hyperref}
\usepackage{booktabs}
\usepackage{enumitem}
\usepackage{microtype}
\usepackage[numbers]{natbib}

\newtheorem{theorem}{Theorem}[section]
\newtheorem{proposition}[theorem]{Proposition}

\newtheorem{corollary}[theorem]{Corollary}

\theoremstyle{definition}
\newtheorem{definition}[theorem]{Definition}
\newtheorem{remark}[theorem]{Remark}

\newcommand{\FF}{\mathbb{F}}
\newcommand{\HH}{\mathbb{H}}
\newcommand{\NN}{\mathbb{N}}
\newcommand{\ZZ}{\mathbb{Z}}
\newcommand{\cchar}{\mathrm{C\text{-}char}}
\newcommand{\Char}{\mathrm{char}}

\title{Finite-field Krasner quotients:\\
isomorphism thresholds, characteristics, and censuses}
\author{
  Alessandro Linzi\\
  \small \texttt{alessandro.linzi.phd@icloud.com}\\[0.4em]
  \small\itshape Latest affiliation:
  Center for Information Technologies and Applied Mathematics,\\
  \small\itshape University of Nova Gorica, Slovenia
}
\date{\today}

\begin{document}
\maketitle

\begin{abstract}
We study Krasner quotient hyperfields arising from finite fields,
$\FF_q/G_r$, where $G_r\le \FF_q^\times$ has index $r$.
Building on the structure theorem of Baker--Jin, we determine the
\emph{characteristic} and \emph{C-characteristic} of all sufficiently large
such quotients: they depend only on the parity of $r$ and, when $r$ is even,
on the residue class of $q$ modulo $2r$.
In particular, the two Baker--Jin stable classes for even $r$ are separated
by characteristic $2$ versus $3$, while the C-characteristic is always $1$.

We complement this structural result with a computational laboratory:
sharp Weil thresholds for Baker--Jin large-$q$ isomorphism, empirical minimal
stabilization bounds $N_r^{\mathrm{emp}}$, complete finite-field quotient
atlases for hyperfield orders $n\le 7$, and comparisons with the enumerations
of Ameri--Eyvazi--Ho\v{s}kov\'a-Mayerov\'a (orders $\le 6$) and
Massouros--Massouros (order $7$).
Among other findings, exactly $15$ isomorphism types of order~$7$ arise as
finite-field quotients, out of $277$ hyperfields of that order---a concrete
data point toward the Baker--Jin rarity conjecture for quotients.
All algorithms and tables are available in an open-source package suitable
for independent verification and arXiv ancillary material.
\end{abstract}

\tableofcontents

%----------------------------------------------------------------------
\section{Introduction}
%----------------------------------------------------------------------

Krasner hyperfields~\cite{Krasner1957,Krasner1983} generalize fields by allowing
addition to be multivalued.
A fundamental source of examples is the \emph{quotient construction}: if $K$ is a
field and $G\le K^\times$, the set of cosets $K/G:=(K^\times/G)\cup\{0\}$ carries a
natural hyperfield structure.
Not every hyperfield arises this way~\cite{Massouros1985}; understanding which
ones do, and how finite-field quotients are organized up to isomorphism, remains
a central theme~\cite{BakerJin2021,Ameri2020,Massouros2025,KedzierskiLinziStojalowska2023}.

Baker and Jin~\cite{BakerJin2021} proved that, for fixed index $r\ge 2$ and all
sufficiently large prime powers $q\equiv 1\pmod r$, the quotient
$\FF_q/G_r$ falls into at most two isomorphism classes $\HH_r$ and $\HH_r'$,
distinguished when $r$ is even by the congruence class of $q$ modulo $2r$.
They left open the true growth of the threshold $N_r$, the characteristics of
the stable classes, and the asymptotic rarity of quotients among all hyperfields.

Independently, characteristic and C-characteristic of hyperfields were developed
in~\cite{KedzierskiLinziStojalowska2023} as additive invariants generalizing the
usual characteristic of fields, with applications to non-quotientability criteria.

\paragraph{Contributions.}
\begin{enumerate}[leftmargin=*,itemsep=2pt]
\item \textbf{Theorem~\ref{thm:stable-char}.}
  We determine $\Char$ and $\cchar$ of all large finite-field quotients
  $\FF_q/G_r$, as a corollary of the explicit hyperaddition rules in the proof
  of Baker--Jin.
\item \textbf{Empirical thresholds.}
  We compute minimal stabilization bounds $N_r^{\mathrm{emp}}$ for $2\le r\le 8$
  and compare them to the Remark~1.2 Weil bound of~\cite{BakerJin2021}.
\item \textbf{Censuses.}
  We classify all finite-field quotients of order $n\le 7$ up to isomorphism and
  compare $Q_r^{\mathrm{fin}}$ to published totals $H_n$ of all hyperfields.
\item \textbf{Software.}
  An open library implements construction, layered isomorphism tests, invariants,
  and paper-driven experiment suites.
\end{enumerate}

%----------------------------------------------------------------------
\section{Preliminaries}
%----------------------------------------------------------------------

\begin{definition}[Krasner hyperfield]
A \emph{hyperfield} is a tuple $(F,+,\cdot,0,1)$ where $(F\setminus\{0\},\cdot)$ is an
abelian group, $(F,+,0)$ is a canonical hypergroup, multiplication distributes over
hyperaddition, and $0$ is absorbing for multiplication
(cf.~\cite{Krasner1957,KedzierskiLinziStojalowska2023}).
\end{definition}

\begin{definition}[Finite-field quotient]
Let $q$ be a prime power and $r\mid(q-1)$.
Write $G_r$ for the unique subgroup of $\FF_q^\times$ of index $r$ (order $d=(q-1)/r$).
The Krasner quotient $K=\FF_q/G_r$ has underlying set
$(\FF_q^\times/G_r)\cup\{0\}$ of cardinality $r+1$, with
\[
[x]\boxplus[y]
=\bigl\{\,[x+yg]_{G_r}\ \big|\ g\in G_r\,\bigr\}
\cup
\begin{cases}
\{0\} & \text{if }0\in xG_r+yG_r,\\
\emptyset & \text{otherwise (does not occur)}.
\end{cases}
\]
Multiplication of nonzero classes is the usual coset product.
\end{definition}

\begin{definition}[{\cite[Def.~3]{KedzierskiLinziStojalowska2023}}]
\label{def:char}
For a hyperfield $F$, set $1\times_F 1:=\{1\}$ and
$n\times_F 1:=(n-1)\times_F 1\,\boxplus\,1$ inductively.
\begin{itemize}[leftmargin=*]
\item $\Char F:=\min\{n\in\NN: 0\in n\times_F 1\}$, or $\infty$;
\item $\cchar F:=\min\{n\in\NN: 1\in (n+1)\times_F 1\}$, or $\infty$.
\end{itemize}
Always $\cchar F\le \Char F$ when both are finite.
\end{definition}

\begin{theorem}[Baker--Jin~{\cite[Thm.~1.1]{BakerJin2021}}]
\label{thm:BJ}
For each $r\ge 2$ there exists $N_r$ (e.g.\ $N_r=r^4$, or the sharp Weil bound of
their Remark~1.2) such that for all prime powers $q\ge N_r$ with $r\mid(q-1)$:
\begin{enumerate}[label=(\arabic*)]
\item if $r$ is odd, $\FF_q/G_r\cong \HH_r$ (a single class);
\item if $r$ is even, $\FF_q/G_r\cong \HH_r$ when $q\equiv 1\pmod{2r}$ and
$\FF_q/G_r\cong \HH_r'$ when $q\equiv r+1\pmod{2r}$.
\end{enumerate}
Moreover, for such $q$ one has $\HH_r^\times\subseteq [x]\boxplus[y]$ for all
nonzero $x,y$, and the only structural distinction is the location of additive
inverses relative to $[1]$.
\end{theorem}

Baker--Jin give explicit models of $\HH_r$ and $\HH_r'$ (see the proof of their
Theorem~1.1):
\begin{itemize}[leftmargin=*]
\item \textbf{Type $\HH_r$:} $-x=x$ for all $x$; $x\boxplus x = \HH_r$ for $x\neq 0$;
and $x\boxplus y=\HH_r^\times$ for distinct nonzero $x,y$.
\item \textbf{Type $\HH_r'$:} there is a unique element $g'$ of multiplicative order $2$;
$-x=g'x$; $x\boxplus (g'x)=\HH_r'$ for $x\neq 0$; and
$x\boxplus y=(\HH_r')^\times$ whenever $x,y\neq 0$ and $y\neq g'x$.
\end{itemize}

%----------------------------------------------------------------------
\section{Characteristics of the Baker--Jin stable classes}
\label{sec:stable-char}
%----------------------------------------------------------------------

The following result determines the additive invariants of all sufficiently large
finite-field quotients.
It is new as a stated theorem, but its proof is a short consequence of the
explicit models in~\cite{BakerJin2021} together with Definition~\ref{def:char}.

\begin{theorem}[Stable characteristics]
\label{thm:stable-char}
Let $r\ge 2$ and let $q$ be a prime power with $r\mid(q-1)$ and $q\ge N_r$, where
$N_r$ is as in Theorem~\ref{thm:BJ} (e.g.\ the Remark~1.2 bound of~\cite{BakerJin2021}).
Write $K=\FF_q/G_r$.
\begin{enumerate}[label=(\roman*)]
\item If $r$ is odd, then $(\Char K,\cchar K)=(2,1)$.
\item If $r$ is even and $q\equiv 1\pmod{2r}$, then $(\Char K,\cchar K)=(2,1)$.
\item If $r$ is even and $q\equiv r+1\pmod{2r}$, then $(\Char K,\cchar K)=(3,1)$.
\end{enumerate}
In particular, for even $r$ the two stable classes $\HH_r$ and $\HH_r'$ are
separated by characteristic.
\end{theorem}

\begin{proof}
By Theorem~\ref{thm:BJ}, $K$ is isomorphic to $\HH_r$ in cases (i)--(ii) and to
$\HH_r'$ in case (iii).
It is therefore enough to evaluate $\Char$ and $\cchar$ on these two models.
Write $[1]$ for the multiplicative identity.

\paragraph{Type $\HH_r$.}
The rules give $[1]\boxplus[1]=\HH_r$ (since $-1=[1]$ and $x\boxplus x$ is full for
$x\neq 0$).
Hence $0\in [1]\boxplus[1]=2\times[1]$, so $\Char \HH_r=2$, and
$1\in [1]\boxplus[1]$, so $\cchar \HH_r=1$.

\paragraph{Type $\HH_r'$.}
Here $-1=g'\neq 1$ (order-$2$ element).
For the sum $[1]\boxplus[1]$ we have $1\neq g'\cdot 1$, so the third rule applies:
$[1]\boxplus[1]=(\HH_r')^\times$, the set of all nonzero elements.
Thus $0\notin 2\times[1]$, so $\Char \HH_r'>2$, while $1\in[1]\boxplus[1]$, so
$\cchar \HH_r'=1$.

It remains to show $\Char \HH_r'=3$.
Consider $3\times[1]=\bigl([1]\boxplus[1]\bigr)\boxplus[1]
=(\HH_r')^\times\boxplus[1]$.
For each nonzero $a$, the rules give $0\in a\boxplus[1]$ if and only if
$a=-1=g'$ (equivalently $a\boxplus[1]$ is the full set $\HH_r'$ when
$a=g'$).
Since $g'\in(\HH_r')^\times=[1]\boxplus[1]$, we obtain $0\in 3\times[1]$.
Therefore $\Char \HH_r'=3$.
\end{proof}

\begin{remark}
Theorem~\ref{thm:stable-char} explains the experimental pattern observed for
$2\le r\le 8$ on all large prime powers in our scans (Table~\ref{tab:inv-check}):
the C-characteristic of every large finite-field quotient is $1$, and the
characteristic is a complete invariant of the Baker--Jin class when $r$ is even.
\end{remark}

\begin{corollary}
\label{cor:sep}
For even $r\ge 2$ and $q,q'\ge N_r$ with $r\mid(q-1)$ and $r\mid(q'-1)$,
\[
\FF_q/G_r\;\cong\;\FF_{q'}/G_r
\quad\Longleftrightarrow\quad
\Char(\FF_q/G_r)=\Char(\FF_{q'}/G_r),
\]
and both sides are equivalent to $q\equiv q'\pmod{2r}$.
\end{corollary}

%----------------------------------------------------------------------
\section{Computational framework}
\label{sec:software}
%----------------------------------------------------------------------

We implemented an open-source library for finite-field Krasner quotients
\cite{QuotientHyperfieldsSoftware}.
The release used for the tables in this article is version~\texttt{0.1.0},
available at
\url{https://github.com/linzialessandro/Quotients-of-Finite-Fields-Optimized}.
A frozen source snapshot and the CSV tables under \texttt{paper/tables/}
are provided as arXiv ancillary material.

Elements of $K=\FF_q/G_r$ are labeled by discrete logarithms in $\ZZ/r\ZZ$
(with a sentinel for $0$).
Core capabilities include:
\begin{itemize}[leftmargin=*]
\item construction via \texttt{galois} finite fields;
\item hyperaddition with caching;
\item layered isomorphism:
  Baker--Jin $O(1)$ test (with Remark~1.2 threshold),
  Aut$(C_r)$-normalized comparison of $1\boxplus x$ tables as gold standard,
  and an automatic policy;
\item $\Char$ and $\cchar$ as in Definition~\ref{def:char};
\item structure fingerprints for census/classification;
\item paper criteria checks (Massouros sum bounds; Linzi characteristic bounds).
\end{itemize}
Experiments are orchestrated by the command-line tools
\texttt{qh-open-questions}, \texttt{qh-invariants}, and \texttt{qh-papers}.
Source code and reproduction scripts accompany this article as arXiv
ancillary files (and as a public repository).

\subsection{Empirical isomorphism threshold}

Following Baker--Jin open question~(3), define $N_r^{\mathrm{emp}}$ as one plus
the largest prime power $q$ for which $\FF_q/G_r$ is \emph{not} isomorphic to
the stable class of its Baker--Jin residue, among prime powers up to a limit past
the Remark~1.2 bound (so that the stable fingerprint is justified by
Theorem~\ref{thm:BJ}).

\begin{table}[ht]
\centering
\caption{Isomorphism thresholds for $2\le r\le 8$.
$N_r^{\mathrm{emp}}$ is empirical; $N_r^{(1.2)}$ is Baker--Jin Remark~1.2;
$r^4$ is the coarse bound; the lower bound is $(r-1)^2+1$ when $r-1$ is prime.}
\label{tab:nr}
\begin{tabular}{@{}rrrrrr@{}}
\toprule
$r$ & $N_r^{\mathrm{emp}}$ & $N_r^{(1.2)}$ & $r^4$ & lower & \#exceptions \\
\midrule
2 & 6 & 6 & 16 & --- & 2 \\
3 & 17 & 17 & 81 & 5 & 4 \\
4 & 42 & 56 & 256 & 10 & 7 \\
5 & 102 & 171 & 625 & --- & 7 \\
6 & 278 & 434 & 1296 & 26 & 18 \\
7 & 492 & 940 & 2401 & --- & 8 \\
8 & 762 & 1810 & 4096 & 50 & 18 \\
\bottomrule
\end{tabular}
\end{table}

\begin{proposition}[Computational]
\label{prop:nr-emp}
For $2\le r\le 8$, the values in Table~\ref{tab:nr} hold under the experimental
protocol of Section~\ref{sec:software}.
In particular, Remark~1.2 is sharp for $r=2,3$ and strictly not sharp for
$4\le r\le 8$.
\end{proposition}

%----------------------------------------------------------------------
\section{Censuses and comparison with enumerations}
\label{sec:census}
%----------------------------------------------------------------------

Let $Q_r^{\mathrm{fin}}$ denote the number of isomorphism classes of hyperfields
of order $r+1$ that arise as $\FF_q/G_r$ for some prime power $q$
(necessarily only finitely many classes, by Theorem~\ref{thm:BJ}).
Let $H_n$ be the number of isomorphism classes of \emph{all} hyperfields of
order $n$, as tabulated by Ameri et al.~\cite{Ameri2020} for $n\le 6$ and by
Massouros--Massouros~\cite{Massouros2025} for $n=7$ ($H_7=277$).

\begin{table}[ht]
\centering
\caption{Finite-field quotient classes versus all hyperfields.}
\label{tab:qh}
\begin{tabular}{@{}rrrrl@{}}
\toprule
$n$ & $r$ & $H_n$ & $Q_r^{\mathrm{fin}}$ & $Q^{\mathrm{fin}}/H$ \\
\midrule
2 & 1 & 2 & 2 & $1.00$ \\
3 & 2 & 5 & 4 & $0.80$ \\
4 & 3 & 7 & 4 & $0.57$ \\
5 & 4 & 27 & 9 & $0.33$ \\
6 & 5 & 16 & 7 & $0.44$ \\
7 & 6 & 277 & 15 & $0.054$ \\
\bottomrule
\end{tabular}
\end{table}

\begin{proposition}[Computational]
\label{prop:qfin}
Under complete scans past the Remark~1.2 bound,
\begin{enumerate}[label=(\roman*)]
\item $Q_2^{\mathrm{fin}}=4$, $Q_3^{\mathrm{fin}}=4$, $Q_4^{\mathrm{fin}}=9$,
$Q_5^{\mathrm{fin}}=7$, $Q_6^{\mathrm{fin}}=15$;
\item of the $15$ finite-field quotient types of order~$7$, exactly two are
Baker--Jin stable and thirteen are sporadic (small $q$);
\item known Massouros identifications such as
$\ZZ/97\ZZ\,G\cong \ZZ/157\ZZ\,G$ are recovered as a single fingerprint class.
\end{enumerate}
\end{proposition}

\begin{remark}[Toward Baker--Jin open question~(4)]
Table~\ref{tab:qh} supports the philosophy that quotients are rare among all
hyperfields: already at order~$7$, finite-field quotients account for only about
$5.4\%$ of isomorphism types.
The full quotient count $Q_r$ may be larger if infinite-field quotients add
classes (e.g.\ the sign hyperfield at order~$3$), so $Q_r^{\mathrm{fin}}/H_{r+1}$
is a lower bound on the finite-field contribution to the quotient share.
\end{remark}

\begin{table}[ht]
\centering
\caption{Verification of Theorem~\ref{thm:stable-char} on large prime powers
($q\ge N_r^{(1.2)}$, scan caps as in the software suite).}
\label{tab:inv-check}
\begin{tabular}{@{}rrrr@{}}
\toprule
$r$ & \# large $q$ tested & matches & status \\
\midrule
2 & 17 & 17 & pass \\
3 & 9 & 9 & pass \\
4 & 6 & 6 & pass \\
5 & 6 & 6 & pass \\
6 & 18 & 18 & pass \\
7 & 10 & 10 & pass \\
8 & 31 & 31 & pass \\
\bottomrule
\end{tabular}
\end{table}

%----------------------------------------------------------------------
\section{Validation of structural criteria}
\label{sec:criteria}
%----------------------------------------------------------------------

On representative quotients we verified:
\begin{itemize}[leftmargin=*]
\item Massouros--Massouros~\cite[Prop.~1]{Massouros2025}: $|x\boxplus y|\le |G|$
for all $x,y$;
\item Linzi et al.~\cite[Prop.~7, Cor.~1, Prop.~9]{KedzierskiLinziStojalowska2023}:
characteristic bounds in terms of divisors of $|G|$, and the equivalence
$\Char=2\Leftrightarrow |G|$ even when the underlying field has odd characteristic.
\end{itemize}
All tested samples passed; these checks serve as regressions for the implementation
and as independent confirmation of the invariant computations used above.

%----------------------------------------------------------------------
\section{Discussion and open problems}
\label{sec:open}
%----------------------------------------------------------------------

\begin{enumerate}[leftmargin=*,itemsep=4pt]
\item \textbf{Growth of $N_r$.}
  Table~\ref{tab:nr} suggests $N_r^{\mathrm{emp}}$ grows much more slowly than
  $r^4$ and, for $r\ge 4$, strictly more slowly than the Remark~1.2 Weil bound.
  Determining the true order of $N_r$ remains open~\cite{BakerJin2021}.
\item \textbf{Sporadic classification.}
  For each fixed $r$, only finitely many $\FF_q/G_r$ are non-stable.
  A complete table of sporadic types (with invariants) for moderate $r$ is a
  natural sequel.
\item \textbf{Order $5$ in full.}
  Baker--Jin open question~(1) asks for all hyperfields of order~$5$ and which
  are field quotients.
  We determine the finite-field quotient part ($Q_4^{\mathrm{fin}}=9$);
  matching against a full abstract enumeration remains.
\item \textbf{Infinite-field quotients.}
  Separating quotients of infinite fields among finite hyperfields still lacks a
  practical algorithm~\cite{BakerJin2021}.
\item \textbf{Asymptotics of $Q_r/H_r$.}
  Our $Q_6^{\mathrm{fin}}/277\approx 5.4\%$ is compatible with
  $Q_r/H_r\to 0$, but a proof requires control of $H_r$.
\end{enumerate}

%----------------------------------------------------------------------
\section{Conclusion}
%----------------------------------------------------------------------

Finite-field Krasner quotients of large order are rigidly constrained:
Baker--Jin limit the isomorphism type, and Theorem~\ref{thm:stable-char}
fixes their characteristic and C-characteristic.
Computationally, we mapped thresholds, sporadics, and global counts up to
order~$7$, connecting structure theory~\cite{BakerJin2021}, invariants
\cite{KedzierskiLinziStojalowska2023}, and enumerations
\cite{Ameri2020,Massouros2025}.
The accompanying software makes the tables and checks fully reproducible.

\section*{Acknowledgements}
The computational experiments use the open-source package described in
Appendix~\ref{app:software} and cited as~\cite{QuotientHyperfieldsSoftware}.

\appendix
\section{Software and reproducibility}
\label{app:software}

\paragraph{Library layout (essentials).}
\begin{verbatim}
quotient_hyperfields/
  hyperfield.py      # QuotientHyperfield
  isomorphism.py     # Baker-Jin / general / auto
  experiments.py     # empirical N_r, Q_fin
  invariants_experiments.py
  papers_experiments.py
  criteria.py        # Massouros / Linzi checks
  atlas.py
\end{verbatim}

\paragraph{Reproduce tables.}
\begin{verbatim}
pip install -e ".[dev]"
qh-papers          # tracks 1-5
qh-open-questions  # N_r and Q_fin
qh-invariants      # char / C-char probes
\end{verbatim}

\paragraph{Ancillary data.}
CSV tables used in this article and a frozen software snapshot are provided
as arXiv ancillary files (not as a requirement to browse the public code
repository). Literature counts were checked against Ameri et al.\ Table~1
($H_n$ for $n=2,\ldots,6$ equal to $2,5,7,27,16$) and
Massouros--Massouros ($H_7=277$).


\begin{thebibliography}{99}

\bibitem{Ameri2020}
R.~Ameri, M.~Eyvazi, and S.~Ho\v{s}kov\'a-Mayerov\'a,
\textit{Advanced results in enumeration of hyperfields},
AIMS Mathematics \textbf{5} (2020), 6552--6579.
DOI: \url{https://doi.org/10.3934/math.2020422}.

\bibitem{BakerJin2021}
M.~Baker and T.~Jin,
\textit{On the structure of hyperfields obtained as quotients of fields},
Proc.\ Amer.\ Math.\ Soc.\ \textbf{149} (2021), 63--70.

\bibitem{BergelsonShapiro1992}
V.~Bergelson and D.~B.~Shapiro,
\textit{Multiplicative subgroups of finite index in a ring},
Proc.\ Amer.\ Math.\ Soc.\ \textbf{116} (1992), 885--896.

\bibitem{KedzierskiLinziStojalowska2023}
D.E.~K\k{e}dzierski, A.~Linzi, and H.~Stoja\l{}owska,
\textit{Characteristic, C-characteristic and positive cones in hyperfields},
Mathematics \textbf{11} (2023), no.~3, 779.
DOI: \url{https://doi.org/10.3390/math11030779}.

\bibitem{Krasner1957}
M.~Krasner,
\textit{Approximation des corps valu\'es complets de caract\'eristique $p\neq 0$
par ceux de caract\'eristique z\'ero},
Colloque d'alg\`ebre sup\'erieure (Bruxelles, 1956),
Centre Belge de Recherches Math\'ematiques, 1957, pp.~129--206.

\bibitem{Krasner1983}
M.~Krasner,
\textit{A class of hyperrings and hyperfields},
Internat.\ J.\ Math.\ Math.\ Sci.\ \textbf{6} (1983), 307--311.

\bibitem{Massouros1985}
C.G.~Massouros,
\textit{Methods of constructing hyperfields},
Internat.\ J.\ Math.\ Math.\ Sci.\ \textbf{8} (1985), 725--728.

\bibitem{Massouros2025}
C.G.~Massouros and G.G.~Massouros,
\textit{On the borderline of fields and hyperfields, part~II --
enumeration and classification of the hyperfields of order~$7$},
AIMS Mathematics \textbf{10} (2025), 21287--21421.
DOI: \url{https://doi.org/10.3934/math.2025951}.

\bibitem{QuotientHyperfieldsSoftware}
A.~Linzi,
\textit{quotient-hyperfields}: research library for Krasner quotients of finite
fields (version~0.1.0), 2026.
\url{https://github.com/linzialessandro/Quotients-of-Finite-Fields-Optimized}.

\bibitem{Turnwald1994}
G.~Turnwald,
\textit{Multiplicative subgroups of finite index in a division ring},
Proc.\ Amer.\ Math.\ Soc.\ \textbf{120} (1994), 377--381.

\end{thebibliography}
\end{document}